\documentclass[a4paper, 11pt]{article}

\usepackage{multirow,array}
\usepackage[american]{babel}
\usepackage{amsmath, amsthm, amssymb, amsfonts}
\usepackage[round]{natbib}
\usepackage{graphicx}
\usepackage[hang, small, bf]{caption}
\usepackage{color}
\usepackage[usenames,dvipsnames,svgnames,table]{xcolor}
\usepackage[colorlinks=true,
            linkcolor=blue,
            urlcolor=blue,
            citecolor=blue]{hyperref}

\usepackage{hyperref}
\usepackage[utf8]{inputenc}                                    
\usepackage[T1]{fontenc}
\usepackage[all,cmtip]{xy}
\usepackage{mathrsfs}
\usepackage{yfonts}

\newtheorem{pro*}{Proposition}[section]

\newtheorem{thm*}[pro*]{Theorem}

\newtheorem{rq*}[pro*]{Remark}

\newtheorem{def*}[pro*]{Definition}

\newtheorem{lem*}[pro*]{Lemma}

\newtheorem{not*}[pro*]{Notations}

\newtheorem{con*}[pro*]{Corollary}

\newtheorem{ex*}[pro*]{Example} 

\numberwithin{equation}{section}

\def\g{\noindent}

\def\p{\vskip4truept}
\def\pp{\vskip 8truept}

\def\pg{\p\g}

\begin{document}

\title {Asymptotic value in frequency-dependent games with separable payoffs: a differential approach.} 
\author{Joseph M. Abdou \footnote{Centre d'Economie de la Sorbonne, Universit\'e Paris 1, Panth\'eon-Sorbonne, 106-112 boulevard de l'H\^opital, 75647 Paris Cedex 13-France; email: abdou@univ-paris1.fr.}, Nikolaos Pnevmatikos \footnote{Universit\'e Paris 2, Panth\'eon-Assas, 4 Rue Blaise Desgoffe, 75006, Paris, France; email: nikolaos.pnevmatikos@u-paris2.fr. This author's research was supported by Labex MME-DII. Part of this research was carried out when the author was working at GERAD of HEC Montr\'eal.}} 
\date{\today} \maketitle
\maketitle

\begin{abstract}
We study the asymptotic value of a frequency-dependent zero-sum game with separable payoff following a differential approach. The stage payoffs in such games depend on the current actions and on a linear function of the frequency of actions played so far. We associate to the repeated game, in a natural way, a differential game and although the latter presents an irregularity at the origin, we prove that it has a value.  We conclude, using appropriate approximations, that the asymptotic value of the original game exists in both the $n$-stage and the $\lambda$-discounted games and that it coincides with the value of the continuous time game.
\end{abstract}

\


\


\

\section*{Introduction} \label{Section 1}

\hspace{0.4cm} Frequency-dependent games are repeated games where the stage payoffs depend on current actions and on frequency of past actions. The interpretation is that in such games, the actions undertaken by the players at each stage may generate externalities, which accumulate as the game unfolds. Stage payoffs may be frequency-dependent over time for several reasons.  For instance, payoffs may change due to learning, habit formation, addiction, or satiation. The class of frequency-dependent games covers a wide variety of applications such as littering and pollution problems, the impact of human activities on other species and more generally on the environment. Frequency-dependent games were introduced by \cite{brenner2003melioration} as a class of dynamic games with stage payoffs depending on the frequency of past actions. To the best of our knowledge, in the context of repeated games, \cite{smale1980prisoner} was the first to study  dynamics that take the past into account. The former games consist in the repetition of a one-shot game in which the stage payoff  depends on the choices of the players at the current stage, as well as on the relative frequencies of actions played at previous stages. An extensive review of this class of games and its applications can be found in \cite{joosten2003games}, where the authors focus mostly on the infinite horizon general frequency-dependent games and in particular derive several Folk-Theorem-like. Moreover, \cite{contou2011contributions} recently studied some aspects of frequency-dependent games. The main insight of this work lies in the fact that on the one hand no uniform value exists even for a one-player game and on the other hand the asymptotic value exists although its convergence is not uniform in the state variable. More precisely, the author considers a particular case of the littering game (\cite{joosten2004strategic}), in which the decision maker has two actions, one that deteriorates the environment, and the other one that preserves it; the littering action is a dominant action in each stage game, but the repeated use of this action  produces a lasting effect on the environment so that all future payoffs are decreased. The author proves in a game with a given length $n \in \mathbb{N}^*$ that the unique optimal strategy consists in using the non littering action from stage $t=1$ up to some stage $t^*(n)$ and then starting to use the littering action until the end of the play. Since the time of switching from  one  action to the other depends on the length of the game, one can prove that the uniform value does not exist in the frequency-dependent control problem.  Nevertheless, the fraction of time $t^*(n)/n$ converges when $n$ goes to infinity so that in particular the asymptotic value $\lim \textbf{V}_n(z)$  exists and is independent of the initial state $z$, although this convergence is not uniform in the state.  

\par In order to characterize the Nash equilibrium payoffs of a long game, the study of the zero-sum case seems to be necessary. In this paper, we investigate the value of a class of two-player zero-sum frequency-dependent games with finite action sets $I$ and $J$ respectively, namely the subclass of frequency-dependent games with separable payoffs\footnote{As far as we know, in general stochastic games, payoffs with separable structure first appeared in \cite{parthasarathy1984stochastic}. The authors study games with transitions independent of the current state.}. This means that the stage payoff is the sum of two parts, one part is derived from the current actions and the other one depends linearly on the frequency of the past actions. This game can be viewed as a stochastic game with countable state space, namely $\mathbb{N}^{I \times J}$ and deterministic transitions. The current state at the $n$-th stage is the \textit{aggregate past matrix}, i.e., it reflects how many times each action profile has been selected in the previous $n-1$ stages. Player 1 maximizes and Player 2 minimizes the average payoff in the first $n$ stages and the game is played under \textit{perfect-monitoring}, meaning that both players know the current state as well as the entire history, i.e., the state visited and the action pair played at each of the preceding stages. Since it is already known from the study of the one-player game that no uniform value exists, our main focus will be the existence of the asymptotic value of this game. We treat in parallel both the average and the discounted case. Note that the convergence being non-uniform in the state variable (see \cite{contou2011contributions}), we cannot rely on the Tauberian theorem of \cite{ziliotto2016tauberian} to deduce the existence of one of the limits from the existence of the other.

The traditional approach to the sequence $(\textbf{V}_n)_{n \in \mathbb{N}^*}$ is through the study of the so-called recursive equation (see \cite{mertens2015repeated}, Theorem 3.2, p.158). However, as in the case of many repeated games, it seems difficult to derive the asymptotic behavior directly from this formula. Therefore, we switch to a differential approach in the sense that we associate to the repeated game, a differential game played over $[0,1] \times \mathbb{R}^{I \times J}$. Indeed, by some heuristic reasoning it is possible to conjecture as a limit of the recursive equation, a hypothetical partial differential equation (PDE) that governs the evolution of the value. It turns out that this is precisely the Hamilton-Jacobi-Bellman-Isaacs (HJBI) equation of some differential game and furthermore that the value of this continuous game is closely related to the value of our repeated game. However, an important difficulty arises due to an irregularity of the payoff function at the origin and it is precisely at the origin where our analysis has to be done.  Everywhere but at the origin, regularity conditions are satisfied by the payoff and dynamics functions and since the Isaacs condition holds true, i.e., lower and upper Hamiltonians coincide, by \cite{evans1984diflerential}\footnote{Existence of the value follows from the standard comparison and uniqueness theorems for viscosity solutions presented in \cite{crandall1983viscosity}.} and \cite{souganidis1999two}, it follows existence of the value in the differential game. Moreover, the value is characterized as the unique viscosity solution in the space of bounded, continuous functions of the HJBI equation with a boundary condition. Despite the irregularity at the origin, we prove existence of the value in the differential game starting at $(0,0)$. In order to compare the values of the repeated game with that of the differential game, we proceed by discretization. The previously mentioned irregularity at the origin makes it impossible to apply the usual methods of approximations. We adapt the methods of \cite{souganidis1999two} for finite horizon differential games (see \cite{bardi2008optimal} for infinite horizon\footnote{The authors prove that under some regularity conditions on the payoff and dynamics functions, the discrete values converge to the values of the continuous time game as the mesh of the discretization tends to $0$. These approximations do not converge in general if the value function is discontinuous.}) so that they can fit our context. We prove that $\textbf{V}_n$ as $n$ tends to infinity, and $\textbf{V}_{\lambda}$ as $\lambda$ goes to zero, both converge to the same limit which is precisely the value of the differential game starting at the origin.
\par In the literature, the use of differential games to study the asymptotic value of a repeated game is not new. A differential approach first appeared in \cite{vieille1992weak} to study weak approachability. An approach similar to ours, has been proposed by \cite{laraki2002repeated} to prove existence of the asymptotic value in $n$-stage and $\lambda$-discounted repeated games with incomplete information on one side. \cite{cardaliaguet2012continuous} achieve a transposition to discrete time games of the numerical schemes used to approximate the value function of differential games via viscosity solution arguments, presented in \cite{barles1991convergence}. The authors prove existence of the asymptotic value in absorbing, splitting and incomplete information games, where convergence is uniform in the state variable.  Our approach differs from all these  literatures by the nature of the state space of the continuous game and chiefly in that, due to the irregularity at the origin in our setting, existence of the value in the continuous game is not straightforward. Since in our model, the state space is countable following their approach would lead us to an infinite dimensional state space in the associated differential game. As a consequence,  the way we associate the differential game to the repeated game is quite different from theirs.

\

\textbf{Structure of the paper}. The remainder of the paper is organized as follows. In Section \ref{section 2}, we give the description of a two-player frequency-dependent zero-sum game and we provide properties of the $n$-stage value function, which will be useful, in the sequel.  In Section \ref{section 3}, starting from the recursive formula satisfied by the value, we heuristically derive a PDE. Then, we define the associated differential game and prove existence of the value in the differential game played over $[0,1]$ and starting at initial state $0$. We then provide its uniformly and $\lambda$-discounted discretized versions. In Section \ref{section 5a}, we conclude by identifying the value of the continuous time game, as the limit value of the $n$-stage and the $\lambda$-discounted frequency-dependent games. In Section \ref{section 6} we conclude.

\

\section{The frequency-dependent game and preliminary results}  \label{section 2}
\pg In this section, we describe our model and we give some preliminary results.
\subsection{Definitions}

Let $I,J$ be finite sets and denote the space of real matrices with $|I|$ rows and $|J|$ columns by $\mathcal{M}^{I \times J}$. The notation $(e_{ij})_{ij}$ stands for the canonical basis in $\mathbb{R}^{I \times J}$, i.e., $e_{ij}=(\delta_i^{i'}\delta_j^{j'})_{i'j'}$, where $\delta_k^{k'}, k \in \{i,j\}$ denotes the Kronecker delta. Let $A=[a_{ij}]$ and $H$  be two elements of $\mathcal{M}^{I \times J}$ and let $z_0 \in \mathcal{Z} :=\mathbb{N}^{I \times J}$. A frequency-dependent zero-sum repeated game with initial state $z_0$ is a dynamic game played by steps as follows:

\pg  At stage $t=1,2,...$, Player 1 and 2 simultaneously and independently choose an action in their own set of actions, $i_t \in I$ and $j_t \in J$ respectively. The stage payoff of Player 1 is given by:
\begin{align*} 
g_t:=g(z_{t-1},i_t,j_t)=a_{i_tj_t} + h(z_{t-1}),
\end{align*}
where $z_t=z_0+e_{i_1j_1}+...+e_{i_{t}j_{t}}$ and for any $z \in \mathcal{Z}$,
\begin{align*}
h({z}):=\begin{cases}\left\langle H,\frac{z}{|z|}\right\rangle, \hspace{0.26cm} z \neq 0\\ 
0, \hspace{1.5cm} z=0,
\end{cases}
\end{align*}
where $\left\langle \cdot, \cdot \right\rangle$ denotes the canonical inner product in $\mathbb{R}^{I \times J}$ and $| \cdot |$ stands for $\left\| \cdot \right\|_1$. The payoff of Player 2 is the opposite of that of Player 1. We assume perfect monitoring of past actions by both players. 

\

\g\textbf{Notation.} In the sequel, we use the notations: $\mathbb{N}^*=\mathbb{N} \setminus \{0\}$ and $\mathbb{R}_+^*=\mathbb{R}_+ \setminus \{0\}$.

\

\subsection{The values of $\Gamma_N(z_0)$ and $\Gamma_{\lambda}(z_0)$}

\hspace{0.4cm}Given $z_0\in \mathcal{Z}$, for any $N \in \mathbb{N}^*$ and any $\lambda \in (0,1)$, we will be interested in the finite $N$-stage and $\lambda$-discounted games of initial state $z_0$, denoted by $\Gamma_N(z_0)$ and $\Gamma_{\lambda}(z_0)$ respectively. A play is given by $\omega:=(i_t,j_t)_{t \in \mathbb{N}^*}$ and the induced rewards in the game of initial state $z_0$, is $\gamma_N(z_0,\omega)=(1/N)\sum_{t=1}^Ng_t$, (resp. $\gamma_{\lambda}(z_0,\omega)=\sum_{t=1}^{\infty}\lambda(1-\lambda)^{t-1}g_t$). Note that, due to the nature of the transition in the state space, announcing the selected moves publicly also reveals the state variable to the players. Therefore, we will denote by $\mathbf{H}_t=\mathcal{Z} \times (I \times J)^{t-1}$ the set of histories at stage $t$ and $\mathbf{H}=\cup_{t \geq 0}\mathbf{H}_t$ will denote the set of all histories. $\Delta(I)$ and $\Delta(J)$ are the sets of mixed moves of Player 1 and Player 2 respectively. A behavioral strategy for Player 1 is a family of maps $\sigma=(\sigma_t)_{t \geq 1}$, such that $\sigma_t: \mathbf{H}_{t} \rightarrow \Delta(I)$. Similarly, a behavioral strategy for Player 2 is a family of maps $\tau=(\tau_t)_{t \geq 1}$, where $\tau_t:\mathbf{H}_{t} \rightarrow \Delta(J)$. $\Sigma$ and $T$ denote the sets of behavioral strategies of Player 1 and Player 2, respectively. Given $z_0 \in \mathcal{Z}$, each strategy profile $(\sigma,\tau)$ induces a unique probability distribution $\mathbb{P}^{z_0}_{\sigma,\tau}$ on the set $\mathcal{Z} \times (I \times J)^{\infty}$ of plays (endowed with the $\sigma$-field generated by the cylinders). $\mathbb{E}^{z_0}_{\sigma,\tau}$ stands for the corresponding expectation. 
\pg We study the games $\Gamma_N(z_0)$ and $\Gamma_{\lambda}(z_0)$ in which the payoff of Player 1 is given by $\gamma_N(z_0,\sigma,\tau)=\mathbb{E}^{z_0}_{\sigma,\tau}\big(\frac{1}{N}\sum_{t=1}^{N}g_t\big)$ and $\gamma_{\lambda}(z_0,\sigma,\tau)=\mathbb{E}^{z_0}_{\sigma,\tau}\left(\sum_{t=1}^{\infty}\lambda(1-\lambda)^{t-1}g_t\right)$ respectively. Existence of the value in $\Gamma_N(z)$ and $\Gamma_{\lambda}(z)$ in mixed strategies follows from the minmax theorem of \cite{neumann1928theorie} and since the game is played under perfect-recall by Kuhn's theorem the value can be achieved by using behavioral strategies. The $N$-stage and $\lambda$-discounted values are given by $\textbf{V}_{N}(z_0) = \sup_{\sigma\in {\Sigma}} \inf_{\tau  \in T} \gamma_N(z_0,\sigma,\tau)$ and $\textbf{V}_{\lambda} (z_0) = \sup_{\sigma\in {\Sigma}} \inf_{\tau  \in T} \gamma_{\lambda}(z_0,\sigma,\tau)$ respectively. By \cite{mertens2015repeated} (Theorem 3.2, p.158) given $(n,\lambda) \in \mathbb{N}^* \times (0,1]$ and a state $z \in \mathcal{Z}$, $\textbf{V}_n$ and $\textbf{V}_{\lambda}$ satisfy the following recursive formulas:
\begin{align}\label{e1}
(n+1) \textbf{V}_{n+1}(z)&= h(z)  + \max\limits_{u \in \Delta(I)} \min\limits_{v \in \Delta(J)} \bigg( \sum\limits_{i,j} u_i v_j \big( a_{ij} + n \textbf{V}_{n}(z + e_{ij}) \big)  \bigg)
\end{align}
\begin{align} \label{eq1'}
{{\bf{V}}_{\lambda}}\left( z \right) &=  \lambda h(z)  + \mathop {\max }\limits_{u \in \Delta(I)} \mathop {\min }\limits_{v \in \Delta(J)}  \bigg(\sum\limits_{i,j}u_iv_j\left( {\lambda{a_{ij}} + (1-\lambda) {{\bf{V}}_{{\lambda}}}\left( {z + {e_{ij}}} \right)} \right) \bigg)
\end{align}

In the remainder of this section, given $n \in \mathbb{N}^*$, we provide a formula for the value of the $n$-stage game. This formula is too complex to allow the study of the limit, nevertheless it sheds a light on the asymptotic behavior of the value.  
\pp

\g\textbf{Notation.} We use the following notations:
\begin{itemize}
\item Given $t \in \mathbb{N}$, let $\Pi_{t}$ denote the subset of the state space $\mathcal{Z}$ defined as follows:
\begin{align*}
\Pi_{t}=\{z \in \mathcal{Z} : |z| = t \}.
\end{align*}

\item We denote the $\max\min$ operator by $\textbf{val}$.

\item For any $(a, p) \in \mathbb{R}^*_+ \times \mathbb{N}^*$, we put :
\begin{align*}
\Lambda_p(a):=\frac{1}{a}+\frac{1}{a+1}...+\frac{1}{a+p-1}=\sum\limits_{k=0}^{p-1}\frac{1}{a+k}.
\end{align*}
\end{itemize}

\

\begin{pro*} \label{pro affineproperty}
For all $n \in  \mathbb{N}^*$, for all $t \in \mathbb{N}^*$, let $K_{n,t} \in \mathcal{M}^{I \times J}$ and $C_{n,t} \in \mathbb{R}$, such that $K_{n,t}=\Lambda_{n}(t) H $ and $C_{n,t}=\sum_{k=1}^{n-1}{\bf val}\big(A+K_{n-k,t+k}\big)$. Then, for all $z \in \Pi_{t}$
\[n\textbf{V}_n(z)=\left\langle K_{n,t}, z \right\rangle + C_{n,t}\]
\end{pro*}
\begin{proof}We proceed by induction on the variable $n$:
\pg For $n=1$, for any $t \in \mathbb{N}^*$ and any $z \in \Pi_t$:
\begin{align*}
\textbf{V}_1(z) =  \left\langle H,\frac{z}{|z|} \right\rangle  + \max\limits_{u \in \Delta(I)} \min\limits_{v \in \Delta(J)}
\left( \sum\limits_{ij} u_i v_j a_{ij} \right).
\end{align*}
Then, $\textbf{V}_1(z)=\left\langle K_{1,t},z \right\rangle + C_{1,t}$, where $K_{1,t}=\frac{H}{t}$ and $C_{1,t}=\textbf{val}(A)$.

\pg The recursive formula  (\ref{e1}) and the induction hypothesis for $n=m$  implies, for all $z \in \Pi_{t}$:
\begin{align*} 
(m + 1)\textbf{V}_{m + 1}(z) 
&= \left\langle \frac{H}{t} + K_{m,t+1},z \right\rangle  + \max\limits_{u \in \Delta(I)}\min\limits_{v \in \Delta(J)}\left( \sum\limits_{ij}u_iv_j\bigg(a_{ij} +  \left\langle K_{m,t+1},e_{ij} \right\rangle + C_{m,t+1}\bigg) \right)
\\
&= \left\langle \frac{H}{t} + K_{m,t+1},z \right\rangle  + \textbf{val}\bigg(A +  K_{m,t+1} \bigg)+ C_{m,t+1}, 
\end{align*}
the middle equality folows from the inner product properties and the fact that $\sum_{ij}u_iv_j=1$, and the last one from the $\textbf{val}$ operator properties. Hence,
\begin{align*}
(m+1)V_{m+1}(z)= \left\langle K_{m+1,t}, z \right\rangle + C_{m+1,t},
\end{align*}
where $K_{m+1,t}=\frac{H}{t} +K_{m,t+1}$ and $C_{m+1,t}=\textbf{val}\big(A+K_{m,t+1}\big) + C_{m,t+1}$, This concludes the proof of the assumption. The rest is routine algebra. Note that $\Lambda_{m+1}(t)=\Lambda_m(t+1)+\frac{1}{t}$.
\end{proof}

\

\begin{con*} \label{con conv of k}
Let $\rho \in (0,1)$. If $N \rightarrow +\infty$ and $\frac{n}{N} \rightarrow \rho$, then
$\lim K_{n,N-n} = -H \ln(1-\rho)$.

\end{con*}
\begin{proof}  Writing $\Lambda_n( N-n) = \sum_{k=1}^{N-1} \frac{1}{k} - \sum_{k=1}^{N-n-1} \frac{1}{k}$, the limit  follows readily from the fact that  the sequence $\sum_{k=1}^n \frac{1}{k}- \ln n$  converges to the Euler constant $\gamma$ when $n$ goes to infinity.
\end{proof}

\

\section{A differential approach} \label{section 3}

\hspace{0.4cm} Given the current stage of the game of total length $N$, the number of stages until the end of the game is denoted by $n$ and in view of Corollary \ref{con conv of k}, the relevant asymptotic filter for the convergence is not that of $N \rightarrow +\infty$ but the one of $\frac{n}{N} \rightarrow \rho \in (0,1)$, where $\rho$ is the fraction of the remaining game. Hence, it seems natural to introduce a continuous version of the dynamic game. To begin with, moving from the recursive formula obtained in (\ref{e1}), we heuristically derive a PDE (Section \ref{section 3.1}). It turns out  that the latter is precisely the HJBI equation of some differential game that we shall define. 

\

\g For any $N \in \mathbb{N}^*$, we define the quotient state space and the uniform partition of $[0,1]$:
\[ \mathcal{Q}_N:=\bigg\{q : q=\frac{z}{N}, z \in \mathcal{Z} \bigg\}, \hspace{2cm} \mathcal{I}_N:=\bigg\{0,\frac{1}{N},...,1\bigg\}.\] 

\

\subsection{The heuristic PDE and the differential game $\mathcal{G}(t,q)$} \label{section 3.1}

We define the function $\Psi_N:\mathcal{I}_N \times \mathcal{Q}_N \rightarrow \mathbb{R}$, such that 
\begin{align} \label{e13}
\Psi_N(t,q):=(1-t)\textbf{V}_n(z),
\end{align}  
where $z=Nq$ and $n=N(1-t)$. Then, $\Psi_N$ satisfies for any $q \in \mathcal{Q}$,
\begin{align} \label{e14} 
\begin{cases}
\Psi_N(t,q)= \frac{h(q)}{N}+ \max\limits_{u \in \Delta(I)} \min\limits_{v \in \Delta(J)} \left(\sum\limits_{i,j} u_iv_j\left(\frac{a_{ij}}{N} +\Psi_N\left(t+\frac{1}{N},q+\frac{e_{ij}}{N}\right) \right)\right), \hspace{0.3cm} t \in \mathcal{I}_N \setminus \{1\}\\
\Psi_N(1,q)=0.
\end{cases}
 \end{align}
The first formula of (\ref{e14}), for any $t \in \mathcal{I}_N\setminus\{1\}$  can be written equivalently:
\begin{align}\label{e15}
0= h(q) + \max\limits_{u \in \Delta(I)} \min\limits_{v \in \Delta(J)} \left(\sum\limits_{i,j} u_iv_j\left(a_{ij} +  N \bigg(\Psi_N\left(t+\frac{1}{N},q+\frac{e_{ij}}{N}\right)-\Psi_N(t,q) \bigg)\right)\right).
\end{align}
When $N \rightarrow + \infty$, we heuristically assume that there exists a sufficiently differentiable function $\Psi: [0,1] \times \mathbb{R}^{I \times J}_{+}\setminus \{0\} \rightarrow \mathbb{R}$ (the limit of $\Psi_N$), which will therefore  satisfy the following (PDE) with boundary condition of (\ref{e14}):
\begin{align} \label{e16}
\begin{cases}
\frac{\partial{\Psi}}{\partial{t}}(t,q) + h(q) + \max\limits_{u \in \Delta(I)}\min\limits_{v \in \Delta(J)}\sum\limits_{i,j}u_iv_j\left(a_{ij} + \frac{\partial{\Psi}}{\partial{q_{ij}}}(t,q)\right)=0, \hspace{0.5cm} (t,q) \in [0,1)\times\mathbb{R}^{I \times J}_{+}\setminus \{0\} ,\\
\Psi(1,q) =  0, \hspace{1cm} q \in \mathbb{R}^{I \times J}_{+}\setminus \{0\} .
\end{cases}
\end{align}

\

\textbf{The differential game}. Given $(t,q) \in [0,1] \times \mathbb{R}^{I \times J}_{+}$, we define a differential zero-sum game, denoted by $\mathcal{G}(t,q)$ starting at time $t$ with initial state $q$. It consists of:
\begin{itemize}
\item The state space $\mathcal{Q}=\mathbb{R}^{I \times J}_{+}$.
\item The time interval of the game $T=[t,1]$.
\item Player 1 uses a measurable control $\tilde{u}:[t, 1] \rightarrow \Delta(I)$ and his control space is $\mathcal{U}_{t}$. Player 2 uses a measurable control $\tilde{v}:[t,1] \rightarrow \Delta(J)$ and his control space is $\mathcal{V}_{t}$. For $t=0$, we use the notation $\mathcal{U}:=\mathcal{U}_0$ and $\mathcal{V}:=\mathcal{V}_0$.
\item If Player 1 uses $\tilde{u}$ and Player 2 uses $\tilde{v}$, then the dynamics in the state space is defined as follows:
\begin{align}\label{ordinary diff}
\begin{cases}\frac{dq}{dt}(s) = \tilde{u}(s) \otimes \tilde{v}(s),\hspace{1cm}s\in(t,1),\\
q(t) = q.\\
\end{cases} 
\end{align}
Clearly, the dynamics is driven by a bounded, continuous function, which is Lipschitz in $q$ and thus (\ref{ordinary diff}) admits a unique solution. 
\item The running payoff at time $s \in [t,1]$ that Player 1 receives from Player 2 is given by $g:\mathcal{Q} \times \Delta(I) \times \Delta(J) \rightarrow \mathbb{R}$ and defined as: 
\begin{align} \label{run payoff}
g(q,u,v) =  h(q) + \big\langle u \otimes v,A \big\rangle
\end{align}
where,
\begin{align*}
h({q}):=\begin{cases}\left\langle H,\frac{q}{|q|}\right\rangle, \hspace{0.25cm} q \neq 0\\ 
0, \hspace{1.49cm} q=0.
\end{cases}
\end{align*}
It is easy to see that $g$ is bounded by $\left\|H\right\|_{\infty} + \left\|A\right\|_{\infty}$ and since $q:[t,1] \rightarrow \mathcal{Q}$ is a differentiable function of time (see (\ref{ordinary diff})), $g$ is differentiable on $\mathcal{Q} \setminus \{0\}$. 

\item The payoff associated to the pair of controls $(\tilde{u},\tilde{v}) \in \mathcal{U}_{t} \times \mathcal{V}_{t}$ that Player 2 pays to Player 1 at time $1$ is given by:
\begin{align}\label{e19}
G(t,q,\tilde{u},\tilde{v}) = \int\limits_{t}^1 g(q(s),\tilde{u}(s),\tilde{v}(s))ds. 
\end{align}
\end{itemize}

\

Following \cite{varaiya1967existence}, \cite{roxin1969axiomatic} and \cite{elliott1972values}, we allow the players to update their controls using non-anticipative strategies. A non-anticipative strategy for Player 1 is a map $\alpha: \mathcal{V}_{t} \rightarrow \mathcal{U}_{t}$ such that for any time $\tilde{t} > t$, 
\begin{align*} 
\tilde{v}_1(s) = \tilde{v}_2(s)\hspace{1cm}\forall s \in [t,\tilde{t}]\hspace{0.5cm}\Rightarrow \hspace{0.5cm}\alpha[\tilde{v}_1(s)] = \alpha[\tilde{v}_2(s)]\hspace{0.5cm}\forall s \in [t,\tilde{t}]. 
\end{align*}
The definition of non-anticipative strategies for Player 2 is analogous. Denote by $\mathcal{A}_{t}$ and $\mathcal{B}_{t}$ the sets of non-anticipative strategies of the players respectively and let us put $\mathcal{A}:=\mathcal{A}_0$ and $\mathcal{B}:=\mathcal{B}_0$. With respect to this notion of strategies, the lower and upper values are defined as follows:
\begin{align*}
W^-(t,q)&: = \sup\limits_{\alpha  \in \mathcal{A}_{t}} \inf\limits_{\tilde{v} \in \mathcal{V}_{t}} G(t,q,\alpha[\tilde{v}],\tilde{v}) \\
W^+(t,q)&: = \inf\limits_{\beta  \in \mathcal{B}_{t}} \sup\limits_{\tilde{u} \in \mathcal{U}_{t}} G(t,q,\tilde{u},\beta[\tilde{u}]). 
\end{align*}
When both functions coincide, we say that the game $\mathcal{G}(t,q)$ has a \textit{value}, denoted by $W(t,q)$.

\

Following \cite{cardaliaguet2000introductiona} and \cite{bardi2008optimal}, the lower and upper hamiltonian functions of the game $\mathcal{G}(t,q)$,
$\mathcal{H}^{\pm}:\mathcal{Q} \times \mathcal{Q}  \rightarrow \mathbb{R}$ are given by:
\begin{align} \label{e20} 
\mathcal{H}^-(\xi,q) &= h(q) + \max\limits_{u \in \Delta(I)} \min\limits_{v \in \Delta(J)}  \big\langle u \otimes v,A + \xi \big\rangle \nonumber
\\
{}
\\
\mathcal{H}^+(\xi,q) &= h(q) + \min\limits_{v \in \Delta(J)} \max\limits_{u \in \Delta(I)}  \big\langle u \otimes v,A + \xi \big\rangle. \nonumber
\end{align}

\

\g \textbf{Notation}. In the sequel, $\mathcal{Q}^*$ stands for $\mathcal{Q} \setminus \{0\}$.

\

Given $(t,q) \in [0,1] \times \mathcal{Q}^*$, we have: (i) $\Delta(I)$ and $\Delta(J)$ are compact sets; (ii) the dynamics in the state space (\ref{ordinary diff}) and the running payoff (\ref{run payoff}) are bounded, continuous in all their variables and Lipschitz in the state variable $q$ functions; (iii) from the minmax theorem in \cite{neumann1928theorie}, it clearly follows that the Isaacs condition, i.e., $\mathcal{H}^-=\mathcal{H}^+$ holds true (see (\ref{e20})). Then, by \cite{evans1984diflerential} and \cite{souganidis1999two}, the differential game $\mathcal{G}(t,q)$ starting at time $t \in [0,1)$ with initial state $q \in \mathcal{Q}^*$ admits a value, denoted by $W(t,q)$. Moreover, the authors characterize\footnote{In the literature, the lower and upper values of a differential game have been first characterized by means of DPP in \cite{elliott1974cauchy}.} the value by means of the Dynamic Programming Principle (DPP). Namely, for all $(t,q) \in [0,1) \times \mathcal{Q}^*$ and all $\delta \in (0,1-t]$, we have:
\begin{align} \label{e21}
W( t,q) = \sup\limits_{\alpha \in \mathcal{A}_t} \inf\limits_{\tilde{v} \in \mathcal{V}_t} \left\{\left\langle H,\int\limits_{t}^{t + \delta} \frac{q(s)}{|q|+s}ds \right\rangle +  \left\langle \int\limits_{t}^{t + \delta} \alpha[\tilde{v}(s)] \otimes \tilde{v}(s)ds ,A \right\rangle  + W^* \right\}, 
\end{align}
where $W^*:= W\big(t + \delta,q(t + \delta)\big)$ with $q(t + \delta) = q + \int\limits_{t}^{t + \delta} \alpha[\tilde{v}(s)] \otimes \tilde{v}( s)ds$. 

\

Furthermore under the preceding assumptions, $W(t,q)$ is the unique solution in the space of real-valued, bounded, continuous functions defined over $[0,1] \times \mathcal{Q}^*$ of the following HJBI equation:
\begin{align} \label{e22}
\begin{cases}\frac{\partial{W}}{\partial{t}}(t,q)+\mathcal{H}\big({\nabla_q} W(t,q),q\big) = 0,\hspace{0.6cm} (t,q) \in [0,1) \times \mathcal{Q}^*,\\
W(1,q)=0, \hspace{0.2cm} q \in \mathcal{Q}^*.
\end{cases}
\end{align}
where, $\mathcal{H}:=\mathcal{H}^-=\mathcal{H}^+$ is the hamiltonian defined earlier. Consequently, one can identify the PDE obtained in (\ref{e16}) with the HJBI equations of (\ref{e22}).

\

\subsection{Existence of the value in $\mathcal{G}(0,0)$} \label{section 4}

\hspace{0.4cm}In this section, we extend the results of the differential game $\mathcal{G}(t,q)$ over the set $[0,1]\times \mathcal{Q}$. Precisely, we show that $W(0,q)$ admits a limit as $q$ tends to $0$ and we further establish that such limit is the value of the game starting at $(0,0)$, which therefore exists. The idea of the proof lies in the consideration of $\frac{\varepsilon}{2}$-optimal strategies in $\mathcal{G}(0,q)$ that are $\varepsilon$-optimal in $\mathcal{G}(0,0)$.

\

\begin{lem*} \label{lem maj}
Let $q \in \mathcal{Q}^*$ and $(\tilde{u},\tilde{v}) \in \mathcal{U} \times \mathcal{V}$. Denote by $q(\cdot)$ and $\tilde{q}(\cdot)$ the trajectories with initial conditions $q(0)=q$ and $\tilde{q}(0)=0$ obtained from (\ref{ordinary diff}). Then, for any $s \in [0,1]$, we have: $\left| h\left(q(s)\right)-h\left(\tilde{q}(s)\right)\right| \leq 2 \left\|H\right\|_{\infty} \hspace{0.05cm} \frac{\left|q\right|}{\left|q(s)\right|}$.
\end{lem*}

\begin{proof} 
If $\tilde q(s) = 0$, we get $h(\tilde q(s))=0$ and thus, $\vert h(q(s) \vert \le \Vert H\Vert_{\infty}$.
For any $s\in [0,1]$, such that $\tilde q(s)\neq 0$, since $q(s) \neq 0$, it is elementary that $\left|\frac{q(s)}{\left|q(s)\right|}-\frac{\tilde{q}(s)}{\left|\tilde{q}(s)\right|}\right| \leq 2\frac{|q(s)-\tilde{q}(s)|}{|q(s)|}$.
Since the controls of the players depend only on the time variable, by the ordinary differential equation (\ref{ordinary diff}), for all $s \in [0,1]$, $\left|q(s) - \tilde q(s)\right| = \left|q\right|$. Then, we obtain $\big|\frac{q(s)}{\left|q(s)\right|}-\frac{\tilde{q}(s)}{\left|\tilde{q}(s)\right|}\big| \leq 2\hspace{0.05cm}\frac{\left|q\right|}{\left|q(s)\right|}$.
Hence, $|h(q(s))-h(\tilde{q}(s))|=\big|\big\langle H , \frac{q(s)}{|q(s)|} - \frac{\tilde{q}(s)}{|\tilde{q}(s)|} \big\rangle \big| \leq 2\left\|H\right\|_{\infty} \hspace{0.05cm} \frac{\left|q\right|}{\left|q(s)\right|}$.
\end{proof}

\
 
\begin{lem*}\label{lem boundstate}
For any $\varepsilon >0$, there exists $\eta \in (0,\frac{1}{4})$, such that for any $q \in \mathcal{Q}^*$ with $|q|=\eta$ and any $(\tilde{u},\tilde{v}) \in \mathcal{U} \times \mathcal{V}$, we have:
\begin{align}\label{boundstate1}
|G(0,q,\tilde{u},\tilde{v}) -G(0,0,\tilde{u},\tilde{v}| <  \varepsilon.
\end{align} 
\end{lem*}

\begin{proof} 
 Let us put $G(q):=G(0,q,\tilde{u},\tilde{v})$ and $G(0):=G(0,0,\tilde{u},\tilde{v})$ and we further denote by $q(\cdot)$ and $\tilde{q}(\cdot)$ the trajectories with initial conditions $q(0)=q$ and $\tilde{q}(0)=0$ obtained from (\ref{ordinary diff}). Then, 
\begin{align*}
|G(q) - G(0)| &=\bigg|\int\limits_0^1 \big(h\left(q(s)\right) - h\left(\tilde{q}(s)\right)\big)ds \bigg|.
\end{align*}
For all $s \in [0,1]$, it holds true that $|q(s)|=|q|+s$. By Lemma \ref{lem maj}, we get:
\begin{align*}
|G(q) - G(0)| &\leq 2 \left\|H\right\|_{\infty}\bigg|\int\limits_0^1  \hspace{0.05cm} \frac{|q|}{|q|+s}ds \bigg| = 2\left\|H\right\|_{\infty}\hspace{0.05cm}|q|\bigg|\int\limits_0^1\frac{ds}{|q|+s}\bigg|\\
&=  2 \left\|H\right\|_{\infty} \bigg|\bigg(\ln\big(1+|q|\big) - \ln\big(|q|\big)\bigg)\bigg|\left|q\right|\\
&=  2 \left\|H\right\|_{\infty} \bigg|\ln\bigg(\frac{1+|q|}{|q|}\bigg)\bigg|\left|q\right|.
\end{align*}
If $|q| < \frac{1}{4}$, we claim that  $\big|\ln(1+|q|) \big| \le |\ln(\vert q \vert)|$. Indeed, since $|q| < \frac{1}{4}$, we have $|q|(1+|q|) < 1$ and thus, $1 < 1+|q| < \frac{1}{|q|}$, so that 
$0\le |\ln(1+|q|)|< |\ln (|q|)|$. Then, for any $q \in \mathcal{Q}^*$ such that $|q| < \frac{1}{4}$, we have:
\begin{align*}
|G(q) - G(0)| \le  4 \left\|H\right\|_{\infty} \big|\ln(\vert q\vert)\big| |q|.
\end{align*}
To conclude, for any $\varepsilon >0$, choose $\eta \in (0,\frac{1}{4})$, such that $\eta |\ln(\eta)| < \frac{\varepsilon}{4\left\|H\right\|_{\infty} +1}$ and the result is immediate.
\end{proof}

\

\begin{thm*}\label{exlimit}
The game $\mathcal{G}(0,0)$ has a value, $W(0,0)=\lim\limits_{q \rightarrow 0} W(0,q)$.
\end{thm*}

\begin{proof}
Let $\varepsilon >0$ and fix $\eta \in (0,\frac{1}{4})$, such that $\eta|\ln(\eta)| < \frac{\varepsilon}{4(4\left\|H\right\|_{\infty}+1)}$. Since the game $\mathcal{G}(0,q)$ admits a value for any $q \in \mathcal{Q}^*$, consider an $\frac{\varepsilon}{4}$-optimal non-anticipative strategy for Player 1 in $\mathcal{G}(0,q)$, where $|q|=\eta$, i.e., a measurable function $\alpha(\cdot)$, such that for all $s \in [0,1]$, $\alpha[\tilde{v}(s)] \in \mathcal{U}$. Then, for any $\tilde{v} \in \mathcal{V}$, we have:
\begin{align*}
G(0,0,\alpha[\tilde{v}],\tilde{v}) > G(0,q,\alpha[\tilde{v}],\tilde{v})- \frac{\varepsilon}{4} > W^-(0,q) - \frac{\varepsilon}{4} - \frac{\varepsilon}{4}
\end{align*}
where the first inequality follows by Lemma \ref{lem boundstate} and the second inequality is due to $\alpha(\cdot)$ being an $\frac{\varepsilon}{4}$-optimal strategy. Reversing the roles of the players and following similar arguments we get that for any $\tilde{u} \in \mathcal{U}$, $G(0,0,\tilde{u},\beta[\tilde{u}]) < G(0,q,\tilde{u},\beta[\tilde{u}]) + \frac{\varepsilon}{4} < W^+(0,q)+\frac{\varepsilon}{4} +\frac{\varepsilon}{4}$, where $\beta(\cdot)$ is an $\frac{\varepsilon}{4}$-optimal non-anticipative strategy of Player 2 in $\mathcal{G}(0,q)$. Since the value exists in $\mathcal{G}(0,q)$ for any $q \in \mathcal{Q}^*$, we have:
\begin{align*}
W^-(0,0)&=\sup\limits_{\alpha \in \mathcal{A}}\inf\limits_{\tilde{v} \in \mathcal{V}}G(0,0,\alpha[\tilde{v}],\tilde{v})> W(0,q) - \frac{\varepsilon}{2}\\
W^+(0,0)&=\inf\limits_{\beta \in \mathcal{B}}\sup\limits_{\tilde{u} \in \mathcal{U}}G(0,0,\tilde{u},\beta[\tilde{u}])< W(0,q) + \frac{\varepsilon}{2}
\end{align*}
and we therefore get $|W^-(0,0)-W^+(0,0)| < \varepsilon$, which proves existence of the value in $\mathcal{G}(0,0)$ since the inequality holds true for any positive $\varepsilon$.
\end{proof}



\

\subsection{The discretized game $\mathcal{G}_{\mathcal{P}}\left(t_0,q_0\right)$} \label{section 3.3.3}

\hspace{0.4cm}In this section, we introduce a discrete version of our differential game  and prove a strong relation between its value and the value of the original repeated game.  For that purpose, we next  consider subdivisions of $[0,1]$ and we define a family of discretized games played on them.

\begin{itemize}
\item For all $t_0 \in [0,1)$, $\mathcal{P}$ stands for any countable subdivision of $[t_0,1]$ and if $\mathcal{P}$ is finite let $\omega_{\mathcal{P}}$ denote the number of intervals of such subdivision, otherwise we put $\omega_{\mathcal{P}}=\infty$.
\item Given $N \in \mathbb{N}^*$, let $\mathcal{P}_N= (t_k^N)_{0 \leq k \leq N}$, where $t_k^N:=\frac{k}{N}$ stands for the uniform subdivision of $[0,1]$ in $N$ intervals. We will also use the notation $\mathcal{P}_N=\big(t_n^N\big) _{0\leq n \leq N}$ for $n=N-k$.
\item Given $\lambda \in (0,1)$, $\mathcal{P}_{\lambda}=(t_k^{\lambda})_{k \geq 0}$ stands for the countable subdivision of $[0,1]$ induced by the discount factor $\lambda$, such that $t_0^{\lambda}=0$, $t_1^{\lambda}=\lambda$, $t_k^{\lambda}:=\lambda+...+\lambda(1-\lambda)^{k-1}$, for $k \geq 1$ and $t_{\infty}^{\lambda}=1$. 
\item By $\pi_k:=t_{k+1}-t_k$ is denoted the $k$-th increment and $|\mathcal{P}|$ stands for the mesh of the subdivision $\mathcal{P}$, i.e., $|\mathcal{P}|=\sup\limits_k |\pi_k|$.
\end{itemize}

Given $\mathcal{P}$, for all $(t_0,q_0) \in [0,1) \times \mathcal{Q}$ we associate to $\mathcal{G}\left(t_0,q_0\right)$ a discrete time game adapted to the subdivision $\mathcal{P}$ denoted by $\mathcal{G}_{\mathcal{P}}(t_0,q_0)$. Such a discrete time game starts at time $t_0$, has initial state $q_0 \in \mathcal{Q}$ and is repeated $\omega_{\mathcal{P}}$ times. At time $t_k \in \mathcal{P}$, both players observe the current state $q_k$ and choose simultaneously and independently actions $u_{k+1}$ and $v_{k+1}$ in $\Delta(I)$ and $\Delta(J)$ respectively. The control sets are denoted by $\Delta(I)^{\omega_{\mathcal{P}}}$ and $\Delta(J)^{\omega_{\mathcal{P}}}$, indicating that players now choose piecewise constant functions defined over the $\omega_{\mathcal{P}}$-times Cartesian product of their corresponding mixed strategy sets. We will use the notation $\hat{u}=(u_k)_{k=1}^{\omega_{\mathcal{P}}}$ and $\hat{v}=(v_k)_{k=1}^{\omega_{\mathcal{P}}}$. The state evolves according to:
\begin{align*}
\hspace{1cm}\begin{cases}q_{k+1}=q_k+\pi_{k} u_{k+1} \otimes v_{k+1}, \hspace{0.5cm}k\geq 0,\\
q_0=q.
\end{cases} 
\end{align*}   
At stage $k$, the expected payoff that Player 1 receives from Player 2 is given by:
\begin{align} \label{e24} 
g(q_{k-1},u_k,v_k) =  h(q_{k-1}) + u_{k}Av_{k} 
\end{align}  
and given $(\hat{u},\hat{v}) \in \Delta(I)^{h_{\mathcal{P}}} \times \Delta(J)^{h_{\mathcal{P}}}$, the total payoff of the game is 
\begin{align} \label{e25}
G_{\mathcal{P}}(q_0,\hat{u},\hat{v})=\sum\limits_{k=1}^{h_{\mathcal{P}}}\pi_{k-1}g(q_{k-1},u_k,v_k).
\end{align}
\par Given $\mathcal{P}$, for all $(t_0,q_0) \in [0,1) \times \mathcal{Q}$, the game $\mathcal{G}_{\mathcal{P}}(t_0,q_0)$ admits a value. Following \cite{friedman1970definition}, the value of the game denoted by $W_{\mathcal{P}}(t_0,q_0)$ is characterized by means of discrete version of the HJBI equations (\ref{e22}):
\begin{align} \label{def of ddpp}
\begin{cases}W_{\mathcal{P}}(t_k,q_k) =\pi_k h(q) + \max\limits_{u \in \Delta(I)}\min\limits_{v \in \Delta(J)}\left(\big\langle \pi_ku \otimes v , A\big\rangle+W_{\mathcal{P}}\big(t_{k+1},q_k+\pi_ku\otimes v\big)\right),\\
W_{\mathcal{P}}(1,q)=0\\
\end{cases}
\end{align}
We will refer to this equation as the \textit{dicrete} Dynamic Programming Principle that will be abbreviated to discrete DPP. 

\

\par In the sequel, we compare the $n$-stage (resp. the $\lambda$-discounted)  game with an appropriate discretization of the differential game.  Note that  they differ essentially in the nature of their  outcome space: While  in  the discretized game  $\mathcal{G}_{\mathcal{P}}\left(t_0,q_0\right)$, the play generated by pure strategies is  deterministic and lives in $\mathbb{R}^{I\times J}$, in the original game $\Gamma_N(z_0)$ (resp. $\Gamma_{\lambda}(z_0)$)  the play generated by behavioral strategies is random and takes its values in a discrete subset of $\mathbb{R}^{I\times J}$.

\subsection{Coincidence of $\Psi_N$ and $W_{\mathcal{P}_N}$}
We first prove that $\Psi_N$ preserves a very similar property to the one satisfied by $\textbf{V}_n$ in Proposition \ref{pro affineproperty} and we then show $\Psi_N = W_{\mathcal{P}_N}$. Recall that  the map $\Psi_N:\mathcal{P}_N \times \mathcal{Q}_N \rightarrow \mathbb{R}$ has been defined by (\ref{e13}) and is characterized by the recursive formula (\ref{e14}). It is then clear that  an extension of $\Psi_N$ 
to a map $\Psi_N: \mathcal{P}_N \times \mathcal{Q} \rightarrow \mathbb{R}$  is obtained  if we define it by the same recursive formula and terminal condition; namely for any $q \in \mathcal{Q}$,
\begin{align} \label{def of psi}
\begin{cases}\Psi_N\left(t_k^N,q\right)= \frac{h(q)}{N}+ \max\limits_{u \in \Delta(I)} \min\limits_{v \in \Delta(J)} \left(\sum\limits_{i,j} u_iv_j\left(\frac{a_{ij}}{N} + \Psi_N\left(t_{k+1}^N,q+\frac{e_{ij}}{N}\right) \right)\right), \hspace{0.15cm} 0 \leq k \leq N-1, \\
\Psi_N(1,q)=0, \hspace{0.4cm} k=N, \hspace{0.1cm} 
\end{cases}
\end{align}

\begin{pro*}\label{aff}
Let $N \in \mathbb{N}^*$. There exists a sequence $(k_{n,s}, c_{n,s}) \in \mathcal{M}^{I \times J} \times \mathbb{R}$ where  $n \in \{0,...,N\}, s \in\mathbb{R}_+ $  such that for all $q \in \mathcal{Q}$, and $n \in \{0,...,N\}$:
\begin{equation}\label{affeq}
\Psi_N(t^N_n,q)=\left\langle k_{n,|q|}, q \right\rangle + c_{n,|q|}. 
\end{equation}
The general terms of the sequence $(k_{n,s})$ are given for $s\in \mathbb{R}_{+}^*$ by:
\begin{align*}
k_{n,s}=\Lambda_{N-n}(N s) H 
\end{align*}

\end{pro*}

\begin{proof} For $s=0$ we take by convention $k_{n,0}=0$ for all $n \in \{0,...,N\}$. Let $q \in \mathcal{Q}^*$, we proceed by  backward induction on the variable $n$: 

\g For $n=N$, $\Psi_N(1,q)=0$ for any $q \in \mathcal{Q}^*$ and thus,  one can take $k_{N, s} =0$ and $c_{N,s}=0$ for all $s>0$.
Assume the result is true for $n=m$, i.e., for all $q \in \mathcal{Q}^*$, there exist $k_{m,s} \in \mathcal{M}^{I \times J}$ and $c_{m,s} \in \mathbb{R}$, such that (\ref{affeq}) is satisfied. For $n=m$, for all $q \in \mathcal{Q}^*$, we get from (\ref{def of psi}):
\pg 
\begin{align*}
\Psi\left( t_m^N,q \right)
&=\left\langle \frac{H}{N}, \frac{q}{|q|} \right\rangle +\max\limits_{u \in \Delta(I)} \min\limits_{v \in \Delta(J)} \left(\sum\limits_{i,j} u_iv_j\left(\frac{a_{ij}}{N} + \left\langle k_{m+1,|q|+\frac{1}{N}}, q+ \frac{e_{ij}}{N} \right\rangle + c_{m+1,|q|+\frac{1}{N}}\right)\right)\\
&=\left\langle \frac{H}{N|q|} + k_{m+1,|q|+ \frac{1}{N}}, q \right\rangle + \max\limits_{u \in \Delta(I)} \min\limits_{v \in \Delta(J)} \left(\sum\limits_{i,j} u_iv_j\left(\frac{a_{ij}}{N} + \left\langle k_{m+1,|q|+\frac{1}{N}}, \frac{e_{ij}}{N} \right\rangle  \right)\right)+\\ &\hspace{11.7cm} +c_{m+1,|q|+\frac{1}{N}}.
\end{align*}
\pg and thus (\ref{affeq}) is satisfied if we put:
\begin{eqnarray*} \label{newcoeff}
&k_{m,s}=\frac{H}{N s}+k_{m+1,s+\frac{1}{N}}\\
&c_{m,s}=\frac{1}{N} {\bf val} \left(A + k_{m+1,s+\frac{1}{N}} \right)+c_{m+1,s+\frac{1}{N}}.
\end{eqnarray*}
This ends the induction.
\end{proof}

\

%
\

\hspace{-0.65cm}\textbf{Notation.} Given $N \in \mathbb{N}^*$ and $q_0 \in \mathcal{Q}$, for all $t \in \mathcal{P}_N$, we define the subset of $\mathcal{Q}$:
\begin{align*}
\mathcal{Q}_{N}(t,q_0)=\big\{q \in \mathcal{Q} : |q|=|q_0|+t \big\}.
\end{align*}

\

\begin{pro*} \label{pro coinc}
Given $N \in \mathbb{N}^*$ and $q_0 \in \mathcal{Q}$, for all $t \in \mathcal{P}_N$ and all $q \in \mathcal{Q}_N(t,q_0)$, 
\begin{align*} 
\Psi_N(t,q)=W_{\mathcal{P}_N}\left(t,q\right).
\end{align*}
\end{pro*}

\begin{proof} 
Both functions share the same terminal condition, i.e $\Psi_N(1,q)=W_{\mathcal{P}_N}(1,q)=0$, for all $q \in \mathcal{Q}$, \big(see (\ref{e13}), and characterization of $W_{\mathcal{P}_N}$ in terms of discrete DPP\big). Thus, it suffices to prove that $\Psi_N$ and $W_{\mathcal{P}_N}$ satisfy the same recursive formula. To that purpose, fix $q_0 \in \mathcal{Q}$ and time $t=\frac{k}{N}$, where $k \in \big\{0,...,N-1\big\}$. By the discrete DPP, it follows that for all $q \in \mathcal{Q}_{N}\left(\frac{k}{N},q_0\right)$,
\begin{align} \label{e28}
W_{\mathcal{P}_N}\left(\frac{k}{N},q\right) = \frac{h(q)}{N}  + \mathop {\max}\limits_{u \in \Delta(I)} \mathop {\min}\limits_{v \in \Delta(J)} \bigg( \left\langle \frac{u \otimes v}{N},A \right\rangle + W_{\mathcal{P}_N}\left(\frac{k+1}{N},q + \frac{u \otimes v}{N}\right) \bigg)
\end{align}
By (\ref{def of psi}), for any $k \in \left\{0,...,N-1 \right\}$,
\begin{align*}
\Psi_N\left(\frac{k}{N},q\right) = \frac{h(q)}{N}   + \mathop {\max}\limits_{u \in \Delta(I)} \mathop {\min}\limits_{v \in \Delta(J)} \left( \sum\limits_{ij} u_iv_j \left( \frac{a_{ij}}{N} + \Psi_N\left(\frac{k+1}{N},q + \frac{e_{ij}}{N}\right) \right) \right), 
\end{align*}
where $q \in \mathcal{Q}_{N}\left(\frac{k}{N},q_0\right)$. Equivalently:
\begin{align*}
\Psi_N\left(\frac{k}{N},q\right) = \frac{h(q)}{N}  + \mathop {\max}\limits_{u \in \Delta(I)} \mathop {\min}\limits_{v \in \Delta(J)} \left( \sum\limits_{ij} u_iv_j \frac{a_{ij}}{N} + \sum\limits_{ij} u_iv_j \Psi_n\left(\frac{k+1}{N}, q + \frac{e_{ij}}{N} \right) \right). 
\end{align*}
By Proposition \ref{aff}, $\Psi_N$ is affine in the state variable $q$ and it thus follows:
\begin{align*}
\Psi_N\left(\frac{k}{N},q\right) = \frac{h(q)}{N}  + \mathop {\max}\limits_{u \in \Delta(I)} \mathop {\min}\limits_{v \in \Delta(J)} \left( \sum\limits_{ij} u_iv_j \frac{a_{ij}}{N} + \Psi_n\left(\frac{k+1}{N}, \sum\limits_{ij} u_iv_j\left(q + \frac{e_{ij}}{N}\right) \right) \right). 
\end{align*}
Hence, due to $\sum_{ij}u_iv_j=1$, we get:
\begin{align*}
\Psi_N\left(\frac{k}{N},q\right) = \frac{h(q)}{N}  + \mathop {\max}\limits_{u \in \Delta(I)} \mathop {\min}\limits_{v \in \Delta(J)} \bigg( \left\langle \frac{u \otimes v}{N},A \right\rangle + \Psi_N\left(\frac{k+1}{N}, q + \frac{u \otimes v}{N} \right) \bigg). 
\end{align*}
This in view of (\ref{e28}), proves that $\Psi_N=W_{\mathcal{P}_N}$.
\end{proof}

\

\subsection{Coincidence of $\Psi_{\lambda}$ and $W_{\mathcal{P}_{\lambda}}$}

\hspace{0.4cm} In view of \cite{contou2011contributions}, the convergence is not uniform in the state variable and as a consequence we cannot use the Tauberian theorem of \cite{ziliotto2016tauberian} to obtain an immediate result on the convergence of the $\lambda$-discounted value, when $\lambda$ tends to $0$. In the case of the $n$-stage value, coincidence between $\Psi_N$ and $W_{\mathcal{P}_N}$ follows  immediately from the preceding paragraph since both functions admit the terminal value zero and satisfy the same recursive equation. Concerning the $\lambda$-discounted game, we  need  to follow a slightly different approach to prove such a coincidence.

\

Recall that  the map $\mathbf{V}_{\lambda}: \mathcal{Z} \rightarrow \mathbb{R}$ has been characterised by the recursive formula (\ref{eq1'}). It is then clear that  an extension of $\textbf{V}_{\lambda}$ 
to a map $\textbf{V}_{\lambda}: \mathcal{Q} \rightarrow \mathbb{R}$  is obtained  if we define it by the same recursive formula; namely, for any $q \in \mathcal{Q}$,
\begin{align}\label{eqext}
{{\bf{V}}_{\lambda}}\left( q \right) &=  \lambda h(q)  + \mathop {\max }\limits_{u \in \Delta(I)} \mathop {\min }\limits_{v \in \Delta(J)}  \bigg(\sum\limits_{i,j}u_iv_j\left( {\lambda{a_{ij}} + (1-\lambda) {{\bf{V}}_{{\lambda}}}\left( {q + {e_{ij}}} \right)} \right) \bigg)
\end{align}
Given $\lambda \in (0,1)$, let us define the function $\Psi_{\lambda}: \mathcal{Q} \rightarrow \mathbb{R}$, such that 
\begin{align} \label{defPsilambda}
\Psi_{\lambda}(q):=\textbf{V}_{\lambda}\left(\frac{q}{\lambda}\right).
\end{align}
By (\ref{eqext}), $\Psi_{\lambda}$ satisfies the following equation:
\begin{align}\label{Psi_lambda}
{\Psi_{\lambda}}\left( q \right) =  \lambda h(q)  + \mathop {\max }\limits_{u \in \Delta(I)} \mathop {\min }\limits_{v \in \Delta(J)} \left( {\sum\limits_{i,j} {{u_i}{v_j}\left( {\lambda{a_{ij}} + (1-\lambda) {{\Psi}_{{\lambda}}}\left( {q + {\lambda e_{ij}}} \right)} \right)} } \right)
\end{align}
By \ref{def of ddpp}, for any $q \in \mathcal{Q}$, we get:
\begin{align}\label{9}
W_{\mathcal{P}_{\lambda}}(0,q)&=\lambda  h(q)+ \max\limits_{u \in \Delta(I)} \min\limits_{v \in \Delta(J)}\left\{\lambda  \sum\limits_{ij}u_iv_ja_{ij}  + W_{\mathcal{P}_{\lambda}}\left(\lambda,q+\lambda u \otimes v\right)\right\} \nonumber \\
&=\lambda  h(q)+ \max\limits_{u \in \Delta(I)} \min\limits_{v \in \Delta(J)}\left\{\lambda  \sum\limits_{ij}u_iv_ja_{ij}  + (1-\lambda)W_{\mathcal{P}_{\lambda}}\left(0,q+\lambda u \otimes v\right)\right\},
\end{align}
where the last equation follows by stationarity of the discounted game $\mathcal{G}_{\mathcal{P}_{\lambda}}(t,q)$. In the sequel, we put $W_{\mathcal{P}_{\lambda}}(q)=W_{\mathcal{P}_{\lambda}}(0,q)$.

\

\g \textbf{Notation}. We use the following notations:





\begin{itemize}
\item The norm $\left\|\cdot \right\|_{\infty}$ of a bounded real-valued function $f$, defined on $\mathcal{Q}$, is \[\left\|f\right\|_{\infty}=\sup\limits_{q \in \mathcal{Q}} |f(q)|.\] 
\item $\mathcal{F}_{\mathcal{B}}$ stands for the set of real-valued bounded functions $f$ defined on $\mathcal{Q}$ with the norm $\left\|\cdot \right\|_{\infty}$. Clearly, $\mathcal{F}_{\mathcal{B}}$ is a Banach space. 
\item $\mathcal{L}_{\mathcal{B}}$ stands for the subspace of $\mathcal{F}_{\mathcal{B}}$, such that if $f \in \mathcal{L}_{\mathcal{B}}$ then there exist bounded and measurable $K:\mathbb{R} \rightarrow \mathbb{R}^{I \times J}$ and $c:\mathbb{R} \rightarrow \mathbb{R}$, such that $f(q)=\left\langle K(|q|),q \right\rangle + c(|q|)$, for any $q \in \mathcal{Q}$.
\item For all $f \in \mathcal{F}_{\mathcal{B}}$ and any $\lambda \in (0,1)$, we define the operator $\Theta_{\lambda}$ as follows:
\begin{align}\label{22}
\Theta_{\lambda}(f)(q)=\lambda\hspace{0.1cm} h(q) + \max\limits_{u \in \Delta(I)}\min\limits_{v \in \Delta(J)} \left(\lambda \sum\limits_{ij}u_iv_j a_{ij} + (1-\lambda)f(q+ \lambda u \otimes v)\right)
\end{align}
and clearly $\Theta_{\lambda}$ admits $W_{\mathcal{P}_{\lambda}}(\cdot)$ as unique fixed point. 
\item For all $f \in \mathcal{F}_{\mathcal{B}}$ and any $\lambda \in (0,1)$, we define the operator $\textbf{T}_{\lambda}$ as follows:
\begin{align}\label{2}
\textbf{T}_{\lambda}(f)(q)=\lambda\hspace{0.1cm} h(q) + \max\limits_{u \in \Delta(I)}\min\limits_{v \in \Delta(J)} \left(\sum\limits_{ij}u_iv_j\big(\lambda a_{ij} + (1-\lambda)f(q+ \lambda e_{ij})\big)\right)
\end{align}
Likewise, $\textbf{T}_{\lambda}$ admits $\Psi_{\lambda}(\cdot)$ as unique fixed point. 
\item Given $\lambda \in (0,1]$, for any $t \in \mathcal{P}_{\lambda}$, we define the following subset of $\mathcal{Q}$:
\begin{align*}
\mathcal{Q}_{\lambda}(t)=\left\{q \in \mathcal{Q} : |q|=t \right\}
\end{align*}
\end{itemize}


\

\

\

\begin{pro*} \label{linop}
If $f \in \mathcal{L}_{\mathcal{B}}$, then $\textbf{T}_{\lambda}(f) \in \mathcal{L}_{\mathcal{B}}$ and $\textbf{T}_{\lambda}(f)=\Theta_{\lambda}(f)$.
\end{pro*}

\begin{proof}
Let $f \in \mathcal{L}_{\mathcal{B}}$. By definition of the operator $\textbf{T}_{\lambda}$, it is easy to see that $\textbf{T}_{\lambda}(f)$ is also bounded.
\g For the rest of the proof, let $\lambda \in (0,1)$ and fix $t \in \mathcal{P}_{\lambda}$. Then, for any $q \in \mathcal{Q}_{\lambda}(t)$,
\begin{align*}
\textbf{T}_{\lambda}(f)(q)&=\lambda\hspace{0.1cm} h(q) + \max\limits_{u \in \Delta(I)}\min\limits_{v \in \Delta(J)} \left(\sum\limits_{ij}u_iv_j\big(\lambda a_{ij} + (1-\lambda)f(q+ \lambda e_{ij})\big)\right)\\
&=\lambda\hspace{0.1cm} \left<H,\frac{q}{|q|} \right> + \max\limits_{u \in \Delta(I)}\min\limits_{v \in \Delta(J)}\left(\sum\limits_{ij}u_iv_j\big(\lambda a_{ij} + (1-\lambda)\left(\left<K(|q|+\lambda), q+ \lambda e_{ij}\right>+c(|q|+\lambda)\right)\big)\right).
\end{align*}
Hence,
\begin{align*}
\hspace{-7.5cm}\textbf{T}_{\lambda}(f)(q)=\left<\frac{\lambda H}{|q|}+(1-\lambda)K(|q|+\lambda),q \right> +
\end{align*}
\[\hspace{3cm}  +\max\limits_{u \in \Delta(I)}\min\limits_{v \in \Delta(J)} \left(\sum\limits_{ij}u_iv_j\left(\lambda a_{ij} + (1-\lambda)\left(\left<K(|q| + \lambda), \lambda e_{ij}\right>+c(|q|+\lambda)\right)\right)\right).\]
It is easy to see that there exist $K': \mathbb{R} \rightarrow \mathbb{R}^{I \times J}$ and $c':\mathbb{R} \rightarrow \mathbb{R}$ such that $\textbf{T}_{\lambda}(f)(q)=\left<K'(|q|),q \right>+c'(|q|)$, where 
\begin{align*}
\begin{cases}
K'(|q|)=\frac{\lambda H}{|q|}+(1-\lambda)K(|q|+\lambda) \in \mathcal{M}^{I \times J}\\
c'(|q|)= \max\limits_{u \in \Delta(I)}\min\limits_{v \in \Delta(J)} \bigg(\sum\limits_{ij}u_iv_j\left(\lambda a_{ij} + (1-\lambda)\left(\left<K(|q| + \lambda), \lambda e_{ij}\right>+c(|q|+\lambda)\right)\right)\bigg) \in \mathbb{R}
\end{cases}
\end{align*}
and it thus, clearly follows that $\textbf{T}_{\lambda}(f) \in \mathcal{L}_{\mathcal{B}}$. \\
For the rest of the proof, since $f \in \mathcal{L}_{\mathcal{B}}$, by equation (\ref{22}) for any $q \in \mathcal{Q}$,
\begin{align*}
\textbf{T}_{\lambda}(f)(q)&=\lambda\hspace{0.1cm} h(q) + \max\limits_{u \in \Delta(I)}\min\limits_{v \in \Delta(J)} \left(\lambda \sum\limits_{ij}u_iv_j  a_{ij} + (1-\lambda)f(q+ \lambda \sum\limits_{ij}u_iv_j  e_{ij})\big)\right)\\
&=\Theta_{\lambda}(f)(q),
\end{align*}
which concludes the proof.
\end{proof}

\

\begin{con*} \label{coincidisc}
Fix $\lambda \in (0,1]$. Then, for any $q \in \mathcal{Q}$, we have $\Psi_{\lambda}(q)=W_{\mathcal{P}_{\lambda}}(q)$.
\end{con*}

\begin{proof}
By Proposition \ref{linop}, since $f\equiv0 \in \mathcal{L}_{\mathcal{B}}$, we get that for any $m \in \mathbb{N}^*$, $\textbf{T}^m(0)=\Theta^m(0)$. It is easy to see that $\textbf{T}_{\lambda}$ and $\Theta_{\lambda}$ are contracting operators and it thus follows $\lim_{m \rightarrow +\infty} \textbf{T}^m_{\lambda}(0)=\lim_{m \rightarrow +\infty} \Theta^m_{\lambda}(0)$. In view of equations (\ref{22}) and (\ref{2}), we conclude that $\Psi_{\lambda}=W_{\mathcal{P}_{\lambda}}$.
\end{proof}

This result will allow us to use approximation schemes for differential games in the subsequent parts of the proof. Since the value of the original game is equal to that of the discretized approximated game, proving convergence of the value of the latter will prove the convergence of the value of the former. One difficulty arises however from the irregularity of the differential game at the origin.

\

\section{Existence of the limit value in $\Gamma_N(z)$} \label{section 5a}

\hspace{0.4cm}In this section, we prove the main result of the paper. We show that the asymptotic values of the $N$-stage and  the $\lambda$-discounted games exist and they are independent of the initial state $z$. We further show that $\lim_{N \rightarrow +\infty}\textbf{V}_N(z) =\lim_{\lambda \rightarrow 0}\textbf{V}_{\lambda}(z)=W(0,0)$. For this purpose we first provide some useful lemmas on the value of the original game $\textbf{V}_N$ (resp. $\textbf{V}_{\lambda}$) and the associated function $\Psi_N$ (resp. $\Psi_{\lambda}$).

\

\begin{lem*} \label{helplemma}
Let $\omega=(i_t, j_t)_{t\in \mathbb{N}^*}$ be a play and $(z_t)_{t \in \mathbb{N}}$ be the process in $\mathcal{Z}$, induced by the initial position $z_0=0$ and the play $\omega$. Then, for any $z \in \mathcal{Z}^*$, we have:
\begin{itemize}
\item for any $N \in \mathbb{N}^*$,
\begin{align*}
\left|\gamma_N(z,\omega)-\gamma_N(0,\omega)\right|   \leq \frac{2\left\|H\right\|_{\infty} |z|}{N}\left( \ln\left(\frac{|z|+N}{|z|}\right) + C \right)
\end{align*}
where $C=2\sup_{N \in \mathbb{N}^*}\varepsilon(N)$ and $\varepsilon(N)$ is a function which goes to zero when $N$ tends to infinity.
\item for any $\lambda \in (0,1]$, such that $\lambda z <1$, there exists $N \in \mathbb{N}^*$, i.e., $N+|z|=\lfloor{\frac{1}{\lambda}}\rfloor$, such that we have:
\begin{align*}
\left|\gamma_{\lambda}(z,\omega)-\gamma_{\lambda}(0,\omega)\right|  \leq  2\left\| H \right\|_{\infty} \lambda |z|  \left(\ln\left(\frac{1}{\lambda |z|}\right) + C_{\lambda}^z \right),
\end{align*}
where $C^z_{\lambda}=(1-\lambda)^{\lfloor{\frac{1}{\lambda}}\rfloor - |z|}+C$. 
\end{itemize}

\begin{proof}
Fix $z \in \mathcal{Z}$ and $N \in \mathbb{N}^*$. Then, we have:
\begin{align*}
\left|\sum\limits_{t=1}^N h(z_t)-h(z+z_t)\right|  &= \bigg|\left\langle H, \frac{z}{|z|} \right\rangle+\sum\limits_{t=2}^{N}\left\langle H,\frac{z_t}{t}-\frac{z+z_t}{|z|+t}\right\rangle  \bigg|\\
&= \bigg|\left\langle H, \frac{z}{|z|} \right\rangle+\sum\limits_{t=2}^{N}\left\langle H,\frac{z_t(\left|z\right|+t)-t(z+z_t)}{t(\left|z\right|+t)}\right\rangle  \bigg|\\
&\leq \left|\left\langle H, \frac{z}{|z|} \right\rangle\right|+\sum\limits_{t=2}^{N}\left|\left\langle H,\frac{z_t\left|z\right|-tz}{t(\left|z\right|+t)}\right\rangle  \right| 
\end{align*}
It follows,
\begin{align*}
\left|\sum\limits_{t=1}^N h(z_t)-h(z+z_t)\right|  &\leq \left\|H\right\|_{\infty}\bigg(1+\sum\limits_{t=2}^{N}\frac{\left|z_t\right|\left|z\right|+t\left|z\right|}{t(\left|z\right|+t)}\bigg)\\
&\leq \left\|H\right\|_{\infty}\bigg(2|z|\sum\limits_{t=1}^{N}\frac{1}{\left|z\right|+t}\bigg)\\
& \leq 2\left\|H\right\|_{\infty} |z|\left( \ln\left(\frac{|z|+N}{|z|}\right) + |\varepsilon(|z|+N)| + |\varepsilon(|z|)| \right),
\end{align*}
where $\varepsilon(x)$ is a function that goes to 0 when $x$ tends to infinity. We put $C:=2 \sup_{N \in \mathbb{N}^*}\varepsilon(N)$ and we thus, conclude the proof of the assertion. To show the second assertion, in a similar way, we obtain:
\begin{align*}
\left| \gamma_{\lambda}(0)-\gamma_{\lambda}(z) \right| &\leq 2\lambda |z| \left\|H\right\|_{\infty}\sum\limits_{t=1}^{\infty}\frac{(1-\lambda)^{t-1}}{\left|z\right|+t}
\end{align*}
Define $N:=\inf\{k \in \mathbb{N}^* : k+|z|+1 > 1/\lambda\}$. Hence, we may write:
\begin{align*}
\left| \gamma_{\lambda}(0)-\gamma_{\lambda}(z) \right| &\leq 2\left\| H \right\|_{\infty} \lambda|z|\left(\sum\limits_{t=1}^{N}\frac{1}{|z|+k} +(1-\lambda)^N \sum\limits_{k=1}^{\infty}\lambda (1-\lambda)^{k} \right)\\
& \leq   2\left\| H \right\|_{\infty} \lambda |z|\left(\ln\left(\frac{|z|+N}{|z|}\right)+ |\varepsilon(|z|+N)| + |\varepsilon(|z|)|+ (1-\lambda)^N  \right)\\
& \leq  2\left\| H \right\|_{\infty} \lambda |z| \left(\ln\left(\frac{1}{\lambda |z|}\right) + C+(1-\lambda)^{\lfloor{\frac{1}{\lambda}}\rfloor - |z|} \right),
\end{align*}
which ends the proof of the Lemma.
\end{proof}
\end{lem*}

\

\begin{pro*} \label{lem same}
Let $z \in \mathcal{Z}$. If $\textbf{V}_N(z)$ converges to some $\ell \in \mathbb{R}$ (resp. $\textbf{V}_{\lambda}(z)$), then for any $\tilde{z} \in \mathcal{Z}$, $\textbf{V}_N(\tilde{z})$ (resp. $\textbf{V}_{\lambda}(\tilde{z})$) converges to the same limit $\ell$.
\end{pro*}

\begin{proof}
Given $N \in \mathbb{N}^*$ ($\lambda \in (0,1)$), fix $z \in \mathcal{Z}^*$ and let us consider the games $\Gamma_N(z)$ (resp. $\Gamma_{\lambda}(z)$) and $\Gamma_N(0)$ (resp. $\Gamma_{\lambda}(0)$). For a pair of behavioral strategies $(\sigma,\tau)$ we denote by  $\mathbb{P}_{\sigma,\tau}$ the probability induced on $(I\times J)^{N}$, (resp. ($I \times J)^{\infty}$). With respect to the probability  $\mathbb{P}_{\sigma,\tau}$, by Lemma \ref{helplemma}, we get:
\begin{align*}
\big|\gamma_N(0,\sigma,\tau)-\gamma_N(z,\sigma,\tau)\big| \le
\frac{1}{N}\left(2\left\|H\right\|_{\infty}|z|\left(\ln\left(\frac{|z|+N}{|z|}\right) + C\right)\right).
 \end{align*}
Since the right hand term is independent of $(\sigma,\tau)$ and the sup-norm of the value function is less than or equal to sup-norm of the payoff function, we get:
\begin{align*}
\big|\textbf{V}_N(z)- \textbf{V}_N(0)\big| \le
\frac{1}{N}\left(2\left\|H\right\|_{\infty}|z|\left(\ln\left(\frac{|z|+N}{|z|}\right) + C\right)\right).
\end{align*}
The conclusion follows by remarking that the right hand side goes to zero when $N \rightarrow \infty$.
\\
For the rest of the proof, fix $\lambda \in (0,1)$. Likewise, by Lemma \ref{helplemma}, we get:
\begin{align*}
\left| \gamma_{\lambda}(0,\sigma,\tau)-\gamma_{\lambda}(z,\sigma,\tau) \right| \leq 2\left\| H \right\|_{\infty} \lambda |z| \left(\ln\left(\frac{1}{\lambda |z|}\right) + C_{\lambda}^z \right).
\end{align*}
Likewise, we obtain:
\begin{align*}
\left|\textbf{V}_{\lambda}(z) - \textbf{V}_{\lambda}(0)\right| \leq 2\left\| H \right\|_{\infty} \lambda |z| \left(\ln\left(\frac{1}{\lambda |z|}\right) + C_{\lambda}^z \right).
\end{align*}
Note that $C_{\lambda}^z$ converges to $1/e$, as $\lambda$ tends to $0$. The result follows since the right hand side goes to zero when $\lambda$ tends to zero.
\end{proof}

\

\

Next theorem provides the main result of the paper. Given $z \in \mathcal{Z}$, we first show that $\textbf{V}_N(z)$ converges to $W(0,0)$ when $N$ tends to infinity and we then prove that $\textbf{V}_{\lambda}(z)$ converges to the same limit, when $\lambda$ goes to 0.

\

\begin{thm*}
For any $z \in \mathcal{Z}$, $\lim\limits_{N \rightarrow +\infty} \textbf{V}_N(z) = \lim\limits_{\lambda \rightarrow 0} \textbf{V}_{\lambda}(z)= W(0,0)$.
\end{thm*}

\begin{proof}
We first prove that for any $z \in \mathcal{Z}$, $\lim_{N \rightarrow +\infty}\textbf{V}_N(z)=W(0,0)$. To that purpose, fix $\varepsilon >0$, choose $\eta \in (0,\frac{1}{4})$ such that,
\begin{align*}
\begin{cases}
\eta < \frac{\varepsilon}{12(\left\|H\right\|_{\infty}(\ln(2)+C))}\\
\eta \ln(\eta) <\frac{\varepsilon}{12\left\|H\right\|_{\infty}+1}
\end{cases}
\end{align*}
and in view of Theorem \ref{exlimit}, we also require $\eta$ to be such that for all $q\in \mathcal{Q}^*$ with $\left|q\right|\le\eta$, $\left|W\left(0,q\right)- W(0,0) \right| < \frac{\varepsilon}{3}$.
Choose some $q_0$ such that $|q_0| = \eta$.
Assumptions on the strategy sets, the dynamics and running payoff functions of Theorem 4.4 in \cite{souganidis1999two} are established in $\mathcal{G}\big(0,q_0\big)$. Accordingly, there exists $\delta >0$, such that for all $|\mathcal{P}| < \delta$, the value $W_{\mathcal{P}}$ converges uniformly on every compact set of $\mathcal{Q}$ to $W$, as the mesh of the discretization $|\mathcal{P}|$ tends to $0$. Fix $N_0=\lfloor\frac{1}{\delta}\rfloor+1$ and associate to $\mathcal{G}\big(0,q_0\big)$, for all $N \geq N_0$, a discrete time game adapted to the subdivision $\mathcal{P}_N$, denoted by $\mathcal{G}_{\mathcal{P}_N}\big(0,q_0\big)$. Then, $\big| W_{\mathcal{P}_N}\left(0,q_0\right) - W\left(0,q_0\right)\big| < \frac{\varepsilon}{3}$.
From Proposition \ref{pro coinc}, $W_{\mathcal{P}_N}\left(0,q_0\right)=\Psi_N(0,q_0)$. By Lemma \ref{helplemma}, for any $z \in \mathcal{Z}$, we have:
\begin{align*}
\left|\textbf{V}_N(z)-\textbf{V}_N(0)\right| \le \frac{2\left\|H\right\|_{\infty}}{N}\left(|z|\left(\ln\left(\frac{|z|+N}{|z|}\right)+C\right)\right).
\end{align*}
There exists $z_0 \in \mathcal{Z}$, such that $z_0=\lfloor Nq_0 \rfloor$ and since $|q_0|=\eta$, we have $|z_0|=N \eta -\rho$ for some $\rho \in (0,1)$. By definition $\Psi_N(0,q_0)=\textbf{V}_N(\lfloor Nq_0 \rfloor)$. Hence, 
\begin{align*}
\big|\Psi_N(0,q_0) - \Psi_N\left(0,0\right)\big| \leq 2\left\|H\right\|_{\infty}\left(\eta-\frac{\rho}{N}\right)\left|\ln\left(\frac{\eta+1-\frac{\rho}{N}}{\eta-\frac{\rho}{N}}\right)+C\right|.
\end{align*}
Since $\eta < \frac{1}{4}$, we have $\ln(1+\eta-\rho/N) < \ln(2)$ and $(\eta-\rho/N)|\ln(\eta - \rho/N)| < \eta|\ln(\eta)|$. As a consequence, we get:
\begin{align*}
\big|\Psi_N(0,q_0) - \Psi_N\left(0,0\right)\big| \leq 2\eta\left\|H\right\|_{\infty}\left(\ln(2)+C+|\ln(\eta)|\right)<\frac{\varepsilon}{3}.
\end{align*}
Therefore, for every integer $N \geq N_0$,
\begin{align*}
\big|\Psi_N\left(0,0\right)-W(0,0)\big| &\leq \big|\Psi_N\left(0,0\right)-\Psi_N\left(0,q_0\right)\big|+\big|\Psi_N\left(0,q_0\right)-W\left(0,q_0\right)\big|+\big|W(0,q_0)-W(0,0)\big|\\
& < \varepsilon.
\end{align*}
From (\ref{e13}), $\Psi_N\left(0,0\right)=\textbf{V}_{N}(0)$. It  follows that $\textbf{V}_N(0) \to W(0,0)$ when $N \to \infty$. In view of Lemma \ref{lem same}, we conclude that for any $z \in \mathcal{Z}$, $\textbf{V}_N(z)$ converges to $W(0,0)$ as $N$ tends to infinity.

\

\pg To prove the assertion on the discounted value, likewise we fix $\varepsilon >0$ and choose $\eta > 0$ such that we have:
\begin{align*}
\begin{cases}
\eta < \frac{\varepsilon}{12(\left\|H\right\|_{\infty}+C_{\lambda}^z)}\\
\eta \ln(\eta) <\frac{\varepsilon}{12 \left\| H \right\|_{\infty}+1}
\end{cases}
\end{align*} 
In view of Theorem \ref{exlimit}, we also require $\eta$ to be such that for all $q\in \mathcal{Q}^*$ with $\left|q\right|\le\eta$, $\left|W\left(0,q\right)- W(0,0) \right| < \frac{\varepsilon}{3}$. Following similar arguments to the ones of the first part of the proof, 
 fix $\lambda_0:=\delta$ and associate to $\mathcal{G}\big(0,q_0\big)$, for all $\lambda \leq \min(\lambda_0,\eta)$, a discrete time game adapted to the subdivision $\mathcal{P}_{\lambda}$, denoted by $\mathcal{G}_{\mathcal{P}_{\lambda}}\big(0,q_0\big)$. Then, $\left| W_{\mathcal{P}_{\lambda}}\left(q_0\right) - W\left(0,q_0\right)\right| < \frac{\varepsilon}{3}$. By Corollary \ref{coincidisc}, $W_{\mathcal{P}_{\lambda}}\left(q_0\right)=\Psi_{\lambda}(q_0)$. By Lemma \ref{helplemma}, for any $z \in \mathcal{Z}$, we have:
\begin{align*}
\left|\textbf{V}_{\lambda}(z) - \textbf{V}_{\lambda}(0)\right| \leq 2\left\| H \right\|_{\infty} \lambda |z| \left(\ln\left(\frac{1}{\lambda |z|}\right) + C_{\lambda}^z \right).
\end{align*}
There exists $z_0 \in \mathcal{Z}$, such that $z_0=\lfloor \frac{q_0}{\lambda}\rfloor$ and since $|q_0|=\eta$, we have: $|z_0|=(\eta / \lambda)-\rho$ for some $\rho \in (0,1)$. By definition $\Psi_{\lambda}(q_0)=\textbf{V}_{\lambda}\left(\lfloor \frac{q_0}{\lambda} \rfloor\right)$. It follows:
\begin{align*}
\left| \Psi_{\lambda}(q_0) - \Psi_{\lambda}(0) \right| 
&\leq \left\| H \right\|_{\infty}  2 (\eta-\lambda \rho) \left|\ln\left(\frac{1}{\eta-\lambda \rho}\right) +C_{\lambda}^z \right|.
\end{align*}
Since $\eta \in (0,\frac{1}{4})$, we have $(\eta - \lambda \rho) |\ln(\eta - \lambda \rho)| < \eta |\ln(\eta)|$. It then follows:
\begin{align*}
\left| \Psi_{\lambda}(q_0) - \Psi_{\lambda}(0) \right| 
&\leq 2\left\| H \right\|_{\infty}  \eta (|\ln(\eta)| +C_{\lambda}^z) < \frac{\varepsilon}{3}.
\end{align*}
Therefore, for any $\lambda \leq \min\{\lambda_0,\eta\}$,
\begin{align*}
\big|\Psi_{\lambda}\left(0\right)-W(0,0)\big| &\leq \big|\Psi_{\lambda}\left(0\right)-\Psi_{\lambda}\left(q_0\right)\big|+\big|\Psi_{\lambda}\left(q_0\right)-W\left(0,q_0\right)\big|+\big|W(0,q_0)-W(0,0)\big|\\
&< \varepsilon.
\end{align*}
From (\ref{defPsilambda}), $\Psi_{\lambda}\left(0\right)=\textbf{V}_{\lambda}(0)$. It follows that $\textbf{V}_{\lambda}(0) \to W(0,0)$ when $\lambda \to 0$. By Lemma \ref{helplemma}, we conclude that for any $z \in \mathcal{Z}$, $\textbf{V}_{\lambda}(z)$ converges to $W(0,0)$ as $\lambda$ tends to zero, which completes the proof of the Theorem.
\end{proof}

\

\section{Conclusion and perspectives} \label{section 6}

\hspace{0.4cm}In this paper we have studied two-player zero-sum frequency-dependent games with separable stage-payoffs and established the convergence of $\textbf{V}_n$ and $\textbf{V}_{\lambda}$ as $n$ tends to infinity and $\lambda$ goes to $0$ respectively, to the value of the associated differential game starting at the origin, $W(0,0)$. 
A natural generalization of our existence result concerns a stage payoff function $g(z,i,j)$ which is assumed to be linear  in $z$ and such that the impacts of the past and that of present actions are not separable, but combine together in some way. Lastly, let us mention that since existence of the asymptotic value in the zero-sum case is established, a study of limits of Nash equilibria payoffs in general-sum frequency-dependent games that leads to some Folk-Theorem-like now seems to be possible. In doing so, one may compare the asymptotic results with the ones obtained for infinite games by \cite{joosten2003games}.


\

\bibliographystyle{plainnat}
\bibliography{biblio}

\end{document}